\documentclass[11pt]{amsart}
\usepackage{amsmath}
\usepackage{amsfonts}
\usepackage{amssymb}
\newtheorem{thm}{Theorem}[section]

\newtheorem{remark}[thm]{Remark}

\newtheorem{prop}[thm]{Proposition}

\newtheorem{ex}[thm]{Example}

\newtheorem{con}[thm]{Conjecture}

\newcommand{\inp}[2]{\left\langle #1, #2 \right\rangle}

\renewcommand{\span}{\mathrm{span}}

\newcommand{\bb}[1]{\mathbb{#1}}
\newcommand{\cl}[1]{\mathcal{#1}}

\newcounter{egcounter}
\begin{document}

%\date{4/7/10}

\title[the Feichtinger conjecture]{the Feichtinger conjecture and reproducing kernel Hilbert spaces}

\author[S.~Lata]{Sneh Lata}
\address{Department of Mathematics, University of Houston,
Houston, Texas 77204-3476, U.S.A.}
\email{snehlata@math.uh.edu}

\author[V.~I.~Paulsen]{Vern I. Paulsen}
\address{Department of Mathematics, University of Houston,
Houston, Texas 77204-3476, U.S.A.}
\email{vern@math.uh.edu}

\begin{abstract} We prove two new equivalences of the Feichtinger conjecture that involve 
reproducing kernel Hilbert spaces. We prove that if for 
every Hilbert space, contractively contained in the Hardy space, each Bessel sequence 
of normalized kernel functions can be partitioned into finitely many Riesz basic sequences, 
then a general bounded Bessel sequence in an arbitrary Hilbert space can be partitioned into 
finitely many Riesz basic sequences. In addition, we examine some of these spaces and prove 
that for these spaces bounded Bessel sequences of normalized kernel functions are finite 
unions of Riesz basic sequences.
%We prove that if Bessel sequences of normalized kernel functions in some ``special'' reproducing kernel Hilbert spaces can be partitioned into Riesz basic sequences, then every general bounded Bessel sequence can be partitioned into Riesz basic sequences. In addition, we examine some of these spaces and show that for these spaces bounded Bessel sequences of normalized kernel functions are finite union of Riesz basic sequences.
\end{abstract}
\maketitle

\section{Introduction}\label{intro} We study the Feichtinger conjecture
in the setting of reproducing kernel Hilbert spaces. 
The Feichtinger conjecture originated in harmonic analysis and currently 
is a topic of high interest as it has been shown to be equivalent to the 
celebrated Kadison-Singer Problem (KSP) \cite{CT}. The Feichtinger
conjecture dates back to at least 2003 
and appeared in print in \cite{CCLV}. There is a significant 
body of work on this conjecture 
\cite{BCHL1, BCHL2, BD, Cas, CT, CV, CCLV, CFTW, Gro}.  

There are several versions of the Feichtinger
conjecture, all of which are equivalent to the Kadison-Singer problem,
but we shall be interested in the version involving Bessel sequences.

\begin{con}{\bf Feichtinger Conjecture (FC)}. Every bounded Bessel sequence in a Hilbert space can be
 partitioned into finitely many Riesz basic sequences.
\end{con}

In this paper, we specialize this conjecture to the case where the
underlying Hilbert space belongs to a special family of reproducing
kernel Hilbert spaces on the unit disk, namely, the Hilbert spaces that are contractively contained in the Hardy space $H^2$ and require, in addition, that the bounded Bessel sequence
consists of normalized kernel functions for a sequence of points in
the disk. One of our results is that this special version of the
Feichtinger conjecture is equivalent to the Feichtinger conjecture.

%In fact, this is a variation of Feichtinger Conjecture, which is also known to be equivalent to 
%KSP. The original conjecture is about partitioning bounded frame sequences into Riesz basic sequences. 

%The equivalence of FC to KSP and hence to many other open problems in 
%various different areas of mathematics shows how difficult is this 
%problem. This conjecture is about testing an entire family of
%sequences and so investigating specific sequences might be a fruitful
%approach. 

Our work is motivated by some work of Nikolski. 
%In \cite{BD, CCLV, Gro} 
%FC is proved for some well-understood sequences and the motivation of our work is rooted in one such 
%result. 
In his lecture at the AIM workshop ``The Kadison-Singer Problem'' in
2006, Nikolski proved that any Bessel sequence
consisting of normalized 
kernel functions in the Hardy space $H^2$ could be partitioned into
finitely many Riesz basic sequences. Later in 2009, Baranov and
Dyakonov  \cite{BD} proved the analogous result
 for two families of 
model subspaces of $H^2.$
%some other well-understood sequences arising
%from kernels in various reproducing kernel Hilbert spaces.

Thus, we were motivated to seek a converse. That is, to find a
sufficiently large family of reproducing kernel Hilbert spaces, so that if one
verify that each Bessel sequence of normalized kernel functions in
those spaces could be partitioned into finitely many Riesz basic sequences,
then that would guarantee the full FC. 

%Since kernel functions are the  
%``building blocks'' in $H^2,$ it motivated us to explore if there is a way 
%to translate the FC totally in terms of kernel functions. We succeed in it and that's 
%what we wish to present here. We prove two new equivalences of the FC. 
%We prove that if every bounded Bessel sequence of normalized kernel 
%functions in a "special" class of reproducing kernel Hilbert spaces 
%can be partitioned into Riesz basic sequences, then the FC is true. 
In addition, we also prove that 
in order to verify the FC it is enough to test a specific family of 
sequences which is related to kernel 
functions in $H^2.$ 
%that carry some additional structure.
To state these equivalences formally we need the following basic notations and terminologies.
 
%Before talking any further about the FC for kernel 
%functions we shall first give some basic notations and terminologies.

Given a set 
$J$ we let $\ell^2(J)$ denote the usual Hilbert space of
square-summable functions on $J$ with canonical orthonormal basis
$\{e_i\}_{i\in J}.$  We let $I_{\ell^2(J)}$ 
denote the identity operator on $\ell^2(J).$ When $J \subseteq \bb N,$
we regard $\ell^2(J) \subseteq \ell^2(\bb N) = \ell^2$ and let $P_J$ denote the orthogonal
projection of $\ell^2(\bb N)$  onto $\ell^2(J).$ 

A set of vectors $\{f_i\}_{i\in J}$ in a Hilbert spaces $\cl H$ is called a {\bf frame} for $\cl H$ 
if there exist constants $A, B>0$ such that 
\begin{equation} 
A\|x\|^2\le \sum_{i\in J}|\inp{x}{f_i}|^2\le B\|x\|^2 
\end{equation}
for every $ x\in \cl H.$ A countable collection $\{f_i\}_{i\in J}$ in a Hilbert space $\cl H$ is called a 
{\bf frame sequence} if it is a frame for $\overline{{\span}\{f_i:i\in J\}}.$ If only the right 
hand side inequality holds in Inequality $(1),$ then $\{f_i\}_{i\in J}$ is called a Bessel set. A countable Bessel set is called a {\bf Bessel sequence}. Thus, every frame sequence is a Bessel sequence. A Bessel 
sequence $\{f_i\}_{i\in J}$ is called {\bf bounded,} if there exists a constant $\delta>0$ such 
that $\|f_i\|\ge \delta$ for every $i\in J.$  Note that a Bessel sequence is always bounded above.

Further, a set $\{f_i\}_{i\in J}$ in a Hilbert spaces $\cl H$ is called a {\bf Riesz basis} 
for $\cl H$ if there exists an orthonormal basis $\{u_i\}_{i\in J}$ for $\cl H$ and an 
invertible operator $S\in B(\cl H)$ such that $S(u_i)=f_i$ for every $i\in J.$ It is 
easy to verify that a countable set $\{f_i\}_{i\in J}$ is a Riesz basis for $\cl
H$ if and only if 
its linear span is dense in $\cl H$ and there exist 
constants $A, B>0$ such that  
$$A\sum_{i\in J}|\alpha_i|^2\le \|\sum_{i\in J}\alpha_if_i\|^2 \le B\sum_{i\in J}|\alpha_i|^2$$ for all 
square summable sets $\{\alpha_i\}_{i\in J}.$ A countable set $\{f_i\}_{i\in J}$ is called a {\bf Riesz basic sequence} 
if it is a 
Riesz basis for $\overline{{\span}\{f_i:i\in J\}}.$ It is well-known that every Riesz basic 
sequence is a frame sequence.  

Given a Bessel set $\{f_i\}_{i\in J}$ in a Hilbert space $\cl H,$ the corresponding 
analysis operator, $F:\cl H\to \ell^2(J)$ defined by
$F(x)=(\inp{x}{f_i})_{i\in J}$ is bounded. It is easy to check 
that $F^*:\ell^2(J)\to \cl H$ is given by $F^*(e_i)=f_i$ for all $ i\in J,$  
and $FF^*=(\inp{f_j}{f_i}).$
% where $\{e_i\}_{i\in I}$ is the canonical orthonormal basis for $l^2(I).$ 
The operators $F^*$ and $FF^*$ are called the {\bf synthesis operator} and the {\bf Grammian} of the set $\{f_i\}_{i\in J},$ respectively. 
Note that:

\begin{itemize}
\item[i.] $\{f_i\}_{i\in J}$ is a Bessel sequence iff the map $F$ is
  bounded iff $FF^*$ is bounded; 
\item[ii.] $\{f_i\}_{i\in J}$ is a frame sequence iff $F:\overline{{\span}\{f_i:i\in J\}}\to \ell^2(J)$ 
is bounded and bounded below; 
\item[iii.] $\{f_i\}_{i\in J}$ is a Riesz basic sequence iff $F:\overline{{\span}\{f_i:i\in J\}}\to \ell^2(J)$ 
is invertible.
%\item If $\{f_i:i\in I\}\subseteq \cl H, \ I\subseteq \bb N$ is a Bessel sequence, then 
%the analysis operator will be a bounded operator from $\cl H$ to $l^2(I).$
\end{itemize}

Henceforth, given a Bessel sequence $\{f_i\}_{i\in J}$ in a Hilbert space $\cl H, \ F$ 
is reserved for the analysis operator from $\cl H$ to $\ell^2(J),$ as defined above.   

We now wish to recall some basic terminology from reproducing kernel Hilbert space theory. 

Recall that a {\bf reproducing kernel Hilbert space(RKHS) on a set X} is a
Hilbert space $\cl H$ of functions on $X$ such that evaluation at any
point in $X$ is a bounded linear functional. The function $K:X \times
X \to \bb C$ given by $K(x,y)
=k_y(x)$ where $f(y) = \langle f, k_y \rangle$ for all $f \in \cl
H$ is called the {\bf
  reproducing kernel for $\cl H.$} The function $k_y, \ y\in X,$ is called 
the {\bf kernel function for the point $y$}. 
We set $\widetilde{k_y} =
\frac{k_y}{\|k_y\|}= \frac{k_y}{\sqrt{K(y,y)}}$ whenever $k_y\ne 0$ and call it the {\bf
  normalized kernel function for the point $y$}.

Let $\bb D$ denote the open unit disk in the complex plane, and let $H^2$ denote 
the familiar Hardy space on $\bb D.$ Recall that it is a reproducing
kernel Hilbert space on $\bb D$ 
with reproducing
kernel $K(z,w)=\frac{1}{1-z\bar{w}}, \  z, w\in \bb D,$ which is
called the ${\rm Szeg\ddot{o}}$ kernel.  

A Hilbert space $\cl H$ is said to be {\bf contractively contained} in $H^2$ if it is a vector subspace of $H^2$ and 
the inclusion of $\cl H$ into $H^2$ is a contraction, that is, 
$\|h\|_{H^2} \le \|h\|_{\cl H}$ for every $h$ in $\cl H.$ Lastly, recall that every Hilbert space that is contractively contained in $H^2$ is a RKHS. 
Given a Hilbert space $\cl H,$ that is contractively contained in $H^2, \ {k}_{z}^{\cl H}$ 
will denote the kernel function for $z\in \bb D$ and 
the corresponding normalized kernel function will be denoted by
$\widetilde{k}_{z}^{\cl H}.$ 

Henceforth, $k_z$ and $\widetilde{k_z}$
shall be reserved to denote the kernel function and the corresponding normalized kernel function in $H^2$ for the point $z\in \bb D$, respectively. 
%${\rm Szeg\ddot{o}}$ kernel. We can now state our main result.
%For simplicity we will call a Hilbert space {\bf contractive} if 
%it is contractively contained in $H^2.$  

%Finally, we wish to recall some reproducing kernel Hilbert spaces which are critical for our work.

%If $\cl H$ is a Hilbert space and $T\in B(\cl H),$ then the range space $\cl R(T)$ 
%is the Hilbert space one obtains by equipping the range of $T$ with the norm, 
%$\|y\|_{\cl R(T)}=\|x\|_{\cl H},$ where $x$ is the unique vector in $\cl K(T)^{\perp}$ satisfying 
%$y=T(x).$ Thus $\cl R(T)$ is a vector subspace of $\cl H,$ and is complete with a different norm.
%In this note we are dealing with range spaces but specific ones. Let %$P\in B(H^2)$ be 
%a positive operator and let $H(P)$ denote the range space $\cl R(P^{1/2}).$ Then $H(P)$ is a 
%vector subspace of $H^2,$ but is also a Hilbert space with this new norm. For notational 
%convenience we shall use $\|\cdot\|_{P}$ instead of $\|\cdot\|_{\cl R(P^{1/2})}$ for 
%the norm on $H(P).$ We shall denote the kernel function for $H(P)$ by $k^P.$ 
%Note that $k^P_w=Pk_w,$ for $w\in \bb D.$ For the normalized kernel function $\frac{k_w^P}{\|k_w^P\|_P}$ 
%we shall use $\widetilde{k}_w^P.$ For details see \cite{AFMP}. Note 
%that when $P\in B(H^2)$ is a positive contraction, then $H(P)$ is a de-Branges space. 

\begin{thm}\label{equi}The following are equivalent:
\begin{itemize}
\item[(i)]every bounded Bessel sequence in a Hilbert space can be partitioned into finitely many Riesz basic 
sequences ({\bf{\em FC}}),
\item[(ii)] every bounded Bessel sequence of the form
  $\{P\widetilde{k}_{z_i}\}_{i\in \bb N},$ where $P\in B(H^2)$ is a positive operator and
  $\{z_i\}_{i \in \bb N}\subseteq \bb D$ is a sequence, can be 
partitioned into finitely many Riesz basic sequences, 

%\item[ii.]for every positive operator $P\in B(H^2),$ every bounded Bessel sequence 
%$\{P\widetilde{k}_{z_i}\}, \ \{z_i\}\subseteq \bb D,$ can be partitioned into finitely many Riesz basic sequences. 

\item[(iii)] every bounded Bessel sequence of the form
  $\{\widetilde{k}_{z_i}^{\cl H}\}_{i\in \bb N},$ where $\cl H$ is contractively
  contained in $H^2$ and  $\{z_i\}_{i\in \bb N}\subseteq \bb D$ is a sequence, 
 can be partitioned into finitely many Riesz basic sequences. 

%\item[iii.] for every de Branges space $\cl H,$ every Bessel sequence $\{\widetilde{k}_{z_i}^{\cl H}\}, \ \{z_i\}\subseteq \bb D,$ of normalized kernel functions in $\cl H$ can be partitioned into finitely many Riesz basic sequences, 
\end{itemize}
\end{thm}

In fact we can also assume a much more restrictive condition on the
sequence $\{z_i\}_{i\in \bb N}$ 
which is that $\{\widetilde{k}_{z_i}\}_{i\in \bb N}$ is a Riesz basic sequence in
$H^2$ or, equivalently, that $\{z_i\}_{i\in \bb N}$ satisfy Carleson's condition
(C). This concept will be defined in the next section. 

\
 
This leads us to formulate the following:

\begin{con}{\bf Feichtinger Conjecture for Kernel Functions (FCKF)}. 
Every Bessel sequence of normalized kernel functions in every
reproducing kernel Hilbert space can be 
partitioned into finitely many Riesz basic sequences. 

%For every 
%de Branges space $\cl H$ and every sequence $\{z_i\}\subseteq \bb D,$ if $\{\widetilde{k}_{z_i}^{\cl H}\}$ is a Bessel sequence of normalized kernel functions in $\cl H,$ then it can be partitioned into finitely many Riesz basic sequences.
\end{con}

From this point forward, we will say that a particular reproducing
kernel Hilbert space $\cl H$ 
satisfies the FCKF if every Bessel sequence of normalized kernel functions 
in $\cl H$ can be partitioned into finitely many Riesz basic sequences.

Thus, the content of our theorem is that not only are the FC and the
FCKF equivalent, but that the FC is equivalent to the FCKF holding
for the family of Hilbert spaces contractively contained in $H^2.$

%From this point forward, we will say that a particular reproducing
%kernel Hilbert space $\cl H$ 
%satisfies the FCKF if every Bessel sequence of normalized kernel functions 
%in $\cl H$ can be partitioned into finitely many Riesz basic sequences. 
%Note that by Theorem \ref{equi}, FC is equivalent to proving that every 
%contractive Hilbert space satisfies the FCKF.

%If in a If every Bessel sequence of normalized kernel functions in a particular de Branges space $\cl H$ can be partitioned into finitely many Riesz basic sequences, then we will say that $\cl H$ satisfy the FCKF. Note that by Theorem \ref{equi}, FC is equivalent to prove that every de Branges space satisfy FCKF.    

%%%%%%%%%%%%%%%%%%%%%%%%%%%%%%%%%%%%%%%%%%%%%%%%%%%%%%%%%%%%%%%%%%%%%%%%%

\section{History}\label{History}
We shall now give a brief history and motivation of our problem. The study of Bessel sequences 
of normalized kernel functions was initiated by Shapiro and Shields in 
1961 \cite{SS}. They analyzed these sequences purely in the context of interpolation problems in the corresponding RKHS. 
In this course, they proved a beautiful result about interpolating sequences in $H^2$ which in late 60's was reformulated by Nikolski and Pavlov \cite{NP1, NP2} as follows: 
%This created a whole new area for research. In the late 60's Nikolski and 
%Pavlov \cite{NP1, NP2} proved the following result:

\begin{thm}\label{NP} A sequence $\{\widetilde{k}_{z_i}\}_{i\in \bb N}$ of normalized kernel functions 
in $H^2$ is a Riesz basic sequence iff there exists a constant $\delta>0$ such that 
$$\begin{array}{lcr} 
(C) \hspace{1 in} & \prod_{i\ne j}{\left|\frac{z_i-z_j}{1-\bar{z_i}z_j} \right|} \ge \delta, & \hspace{.8 in} j=1,2,\dots 
\end{array}$$ 
\end{thm}

%The condition $(C)$ is known as Carleson's condition. 
%In \cite{Car} Carleson showed 
%that if $\{z_i\}$ is an interpolating sequence, then it satisfies 
%$$\sum(1-|z_i|)|f(z_i)|< \infty $$ for all $\ f\in H^1$ and hence satisfies the Blaschke 
%condition $\sum(1-|z_i|)< \infty.$ 

In the late 70's, independent of the work of Nikolski and Pavlov, McKenna was also studying 
kernel functions. In \cite{McK} McKenna proved some partial converses to Shapiro and Shields results 
\cite{SS} and thereby brought some more insight to the area. In 
particular, he proved 
the following interesting result:

\begin{thm}\label{mck} Let $\{\widetilde{k}_{z_i}\}_{i\in \bb N}$ be a Bessel sequence of normalized kernel functions 
in $H^2.$ Then $\{z_i\}_{i\in \bb N}$ can be partitioned into finitely many subsequences each of which 
satisfies the condition $(C).$ 
\end{thm}

Nikolski gave a completely different proof of the above theorem which he included in \cite{Nik1}. 
%All this work was entirely about the study of interpolation problems in  
%reproducing kernel Hilbert spaces. But then 
The FC motivated  Nikolski to combine the above two results as follows: 

\begin{thm}\label{NM}Every Bessel sequence $\{\widetilde{k}_{z_i}\}_{i\in \bb N}$ of normalized kernel 
functions in $H^2$ can be partitioned into finitely many Riesz basic sequences.
\end{thm}

Thus, Theorem \ref{NM} shows that $H^2$ satisfies the FCKF.  This introduced methods from 
reproducing kernel Hilbert space theory to the FC.
% which has already been proved to be a powerful tool in many different areas of pure as well as applied mathematics. 

%%%%%%%%%%%%%%%%%%%%%%%%%%%%%%%%%%%%%%%%%%%%%%%%%%%%%%%%%%%%%%%%%%%%%%%%%%

\section{Preliminary Results}

%\noindent In this section we will present the proof of Theorem \ref{equi}, 
%which will give two new equivalences of
%the FC. 
%It is interesting, as it reduces the question about general Bessel sequences to 
%sequences which have ``more structure'' at our disposal. We start with some preliminary results.

%\begin{prop}\label{a} Let $\{f_i\}$ be a Riesz basis for a Hilbert space $\cl{H}$.
%Then  $\{f_i\}$  is a frame for $\cl{H}.$
%\end{prop}
We begin by recording a few elementary observations that we shall use 
later. Recall that a sequence $\{f_i\}_{i\in J},$ in a Hilbert space 
$\cl H$ is a Bessel sequence iff 
$FF^*=(\inp{f_j}{f_i})\in B(\ell^2(J)),$ that is iff its Grammian is bounded.
The following give characterization of other properties that we shall need in terms of Grammians. 
%Notation: Given a bounded operator $T$ from a Hilbert space $\cl H$ 
%into a Hilbert space $\cl K,$ we let $Ran(T)$ denote the range of $T,$ equipped with the norm of $\cl K.$

\begin{prop}\label{riesz1}Let $\{f_i\}_{i\in J}\subseteq \cl H.$ Then $\{f_i\}_{i\in J}$ is a 
Riesz basis for $\cl H$ iff it is 
a Bessel sequence with closed linear span equal to $\cl H,$ and there exists a constant $c>0$ 
such that $FF^*\ge cI_{\ell^2(J)}.$
\end{prop}

\begin{prop}\label{riesz2}A sequence $\{f_i\}_{i\in J}$  in a Hilbert 
space $\cl H$ is a Riesz basic sequence iff it is 
a Bessel sequence and there exists a constant $c>0$ such that $FF^*\ge cI_{\ell^2(J)}$ 
\end{prop}

Thus, we get the following reformulation of the FC. 

\begin{prop}\label{FC-equi}A Bessel sequence $\{f_i\}_{i\in J}$ can be partitioned into $n$ 
Riesz basic sequences iff there exist a partition 
$A_1,\dots,A_n$ of $J$ and constants $c_1,\dots,c_n>0$
such that $ P_{A_i}FF^*P_{A_i}\ge c_iP_{A_i}$ for all $ 1\le i \le n.$
\end{prop}

\indent From now on, whenever a sequence in a Hilbert space can be partitioned into finitely 
many Riesz basic sequences we will say that it satisfies the FC. 
%{\em Feichtinger conjecture}. 

\begin{prop}\label{r_equi}Let $\{f_i\}_{i\in J}\subseteq \cl H$ and 
$\{g_i\}_{i\in J}\subseteq \cl K$, $J\subseteq \bb N$ be 
two sequences such that $$(\langle{f_j,f_i}\rangle)=D(\langle{g_j,g_i}\rangle) D^*,$$ where $D$ is 
an invertible, diagonal operator in $B(\ell^2(J)).$ Then: 
\begin{itemize}
\item[(i)]$\{f_i\}_{i\in J}$ is a Bessel sequence iff $\{g_i\}_{i\in J}$ is a Bessel sequence,  
\item[(ii)]$\{f_i\}_{i\in J }$ is a frame sequence iff $\{g_i\}_{i\in J}$ is a frame sequence,
\item[(iii)]$\{f_i\}_{i\in J}$ satisfies the FC iff $\{g_i\}_{i\in J}$ satisfies the FC.
\end{itemize}
\end{prop}

\begin{proof}Note that (i) and (iii) are immediate consequences of the
  Grammian characterizations of these properties. To prove (ii) note
  that a Bessel sequence $\{f_i\}_{i \in J}$ is a frame sequence iff the Grammian is
  bounded below on the orthogonal complement of its kernel.
 \end{proof}

While the following result is not necessary for the development of any
of our further results, it does serve to explain why we have chosen to
study the Bessel sequence version of the FC in this context instead of
the frame version. Given a bounded operator $T\in B(\cl H,\cl K)$ we shall denote the kernel of $T$ by $Ker(T).$ 

%The following result gives a glimpse of the rich structure of kernel functions in $H^2.$
%In addition, it also indicates why we have focussed on the Bessel version of the FC.
  
\begin{thm}\label{fr} A sequence $\{\widetilde{k}_{z_i}\}_{i\in J}$ of normalized kernel 
functions in $H^2$ is a frame sequence iff 
it is a Riesz basis basic sequence. Moreover, in this case there is no other kernel function 
in the closed linear span of $\{\widetilde{k}_{z_i}\}_{i\in J}.$
\end{thm}

%\begin{thm}\label{fr} A sequence $\{\widetilde{k}_{z_i}\}_{i\in I}$ of normalized kernel functions in $H^2$ is a frame for a closed subspace $\cl H$ of $H^2$ iff 
%it is a Riesz basis for $\cl H$. Moreover, in this case there is no other kernel function in $\cl H$. 
%\end{thm}

\begin{proof}Let $\cl H$ be the closed linear span of $\{\widetilde{k}_{z_i}\}_{i\in J}$ in $H^2.$ 
If $\{\widetilde k_{z_i}\}_{i\in J}$ is a Riesz basis for $\cl H,$ then clearly it is a frame for $\cl H$. 
To prove the converse, suppose $\{\widetilde k_{z_i}\}_{i\in J}$ is a frame 
for $\cl H$. Then the analysis operator $F:\cl H \rightarrow \ell^2(J),$ given by 
$F(x)=(\langle{x, \widetilde k_{z_i}}\rangle)$ is bounded and $F^*$ is onto. 
Hence, to prove $\{\widetilde k_{z_i}\}_{i\in J}$ is a Riesz basis, it is enough to prove that $F^*$ is  one-to-one.  

To this end, let $\{\lambda_i\}_{i\in J}\in Ker(F^*)$. 
Then $\sum_{i\in J}{\lambda_i \widetilde k_{z_i}}=0,$ which implies that 
$\langle f,\sum_i{\lambda_i \widetilde k_{z_i}}\rangle =0$  for all $f\in H^2,$ which further implies 
that $\sum_{i\in J}{\bar{\lambda}_i\frac{f(z_i)}{\|k_{z_i}\|}}=0$  for all $f\in H^2.$
%\begin{eqnarray*} \sum_i{\lambda_i \widetilde k_{z_i}}=0
% &{\rm implies}& \langle f,\sum_i{\lambda_i \widetilde k_{z_i}}\rangle =0 \ \forall \ f\in H^2\\
%&\Rightarrow& \sum_i{\bar{\lambda}_if(z_i)}=0 \ \forall \ f\in H^2.\\
%\end{eqnarray*}

Note that, $\{\widetilde k_{z_i}\}_{i\in J}$ is a frame and so is a Bessel sequence. Therefore, by 
Theorem~\ref{mck}, $\{z_i\}_{i\in J}$ is a finite union of sets
satisfying (C) and hence satisfies the Blaschke condition. 
Let $f_j$ denote the Blaschke product with zeroes at $\{z_i:i\ne j\}$. 
Then each $f_j$ is in $\ H^2$ and so 
$\sum_{i\in J}{\bar{\lambda}_i}\frac{f_j(z_i)}{\|k_{z_i}\|}=0$
for all $j\in J.$ This forces, $\lambda_j=0$ for all $j\in J.$
Thus, $Ker(F^*)=0$. So, $F^*$ is an invertible operator and 
hence $\{\widetilde k_{z_i}\}_{i\in J}$ is a Riesz basis for $\cl H$.

The moreover part follows by observing that since $\{z_i\}_{i \in J}$
is a Blaschke sequence, the corresponding Blaschke product $B \in H^2$ is
orthogonal to the span of the kernel functions and $B(w) \ne 0$ at any
point not in the set.
\end{proof}

\section{Proof Of The Main Theorem}

The following theorem is a stepping-stone to our main result. 

%It associates a general bounded Bessel sequence to a sequence of kernel functions 
%in $H^2$ and thereby brings kernel functions into play. Thus, 
%kernel functions in $H^2$ will prove to be very crucial objects in our study of the FC. 

\begin{thm}\label{st}Fix a sequence $\{z_i\}_{i\in \bb N}$ in $\bb D$ so that $\{\widetilde{k}_{z_i}\}_{i\in \bb N}$ 
is a frame sequence in $H^2.$ Let $Q\in B(\ell^2)$ be a positive operator such that there exists a constant $\delta>0$ with   
$\langle{Q e_i,e_i}\rangle\ge\delta$ for each $i.$ Then there exists a positive operator $P\in B(H^2)$ 
such that $$Q=\left({\langle{P\widetilde{k}_{z_j},P\widetilde{k}_{z_i}}\rangle}\right)$$ 
with $\|P\widetilde{k}_{z_i}\|^2\ge \delta$ for all  $i.$ 
\end{thm}
 
\begin{proof}Let $\cl H$ be the closed linear span of $\{\widetilde{k}_{z_i}\}_{i\in \bb N}$ 
in $H^2.$ Then, by Theorem \ref{fr}, $\{\widetilde{k}_i\}_{i\in \bb N}$ 
is a Riesz basis for $\cl H,$ since 
$\{\widetilde{k}_{z_i}\}_{i\in \bb N}$ is a frame for $\cl H.$ So, the analysis operator
$F:\cl H\rightarrow \ell^2,$ given by $F(x)=(\langle{x,\widetilde{k}_{z_i}\rangle})$ is invertible with 
$F^*(e_i)=\widetilde{k}_{z_i}$ for each $i$. Set $R=F^{-1}Q(F^{-1})^*.$ Then $R:\cl H\rightarrow \cl H$ 
is a positive, bounded operator. We now extend $R$ to $H^2$ by defining it be $0$ on ${\cl H}^\perp.$  
We claim that $P=R^{1/2}$ satisfies the required conditions. To prove the claim, we fix $i,\ j\in \bb N,$ 
and consider
\begin{eqnarray*}
\langle{P\widetilde{k}_{z_j},P\widetilde{k}_{z_i}}\rangle &=& \langle{R^{1/2}\widetilde{k}_{z_j},R^{1/2}\widetilde{k}_{z_i}}\rangle\\
&=&\langle{R\widetilde{k}_{z_j},\widetilde{k}_{z_i}}\rangle\\
&=&\langle{Q(F^{-1})^*\widetilde{k}_{z_j},(F^{-1})^*\widetilde{k}_{z_i}}\rangle\\
&=&\langle{Qe_j, e_i\rangle}
\end{eqnarray*}
Hence, 
$ Q=\left(\langle {P\widetilde{k}_{z_j},P\widetilde{k}_{z_i}}\rangle\right).$ 
Also, as obtained above 
$\|P\widetilde{k}_{z_i}\|^2 =\langle{P\widetilde{k}_{z_i},P\widetilde{k}_{z_i}}\rangle = \langle{Qe_i,e_i}\rangle \ge \delta$ for all  $i.$ 
This completes the proof.
\end{proof}

\begin{remark} Note that in Theorem~\ref{st}, we have a great deal of freedom in the choice of the frame sequence 
$\{\widetilde{k}_{z_i}\}_{i\in \bb N}$ and hence on the Hilbert space $\cl H=\overline{{\span}\{\widetilde{k}_{z_i}:i\in \bb N\}},$ 
and also on the behavior of 
$P$ on ${\cl H}^{\perp}.$
\end{remark}

We are now ready to give the proof of our main theorem. But, before
proving the theorem we first recall a characterization of the Hilbert
spaces contractively contained in the Hardy space $H^2$ found in the work of
Sarason\cite{Sa}, which is 
very crucial for our proof as it reveals their connection with positive 
contractions on $H^2.$ 

Let $\cl H$ be a Hilbert space that is contractively contained in the
Hardy space $H^2$ with norm $\|\cdot\|_{\cl H}$. Let $T:\cl H\to H^2$ 
be the inclusion map, then $T$ and $T^*:H^2\to \cl H$ are both contractions. Thus, $P=TT^*$ is 
a bounded, positive contraction in $B(H^2).$ This gives rise to another Hilbert space, the 
{\bf range space} $\cl R(P^{1/2}),$ which one obtains by equipping the range of $P^{1/2}$ with 
the norm, $\|y\|_{P}=\|x\|_{H^2},$ where $x$ is the unique vector in the orthogonal complement 
of the kernel of $P^{1/2}$ such that $y=P^{1/2}x.$ 
One has that $\cl H=\cl R(P^{1/2})$ as sets and the two norms coincide. Thus, 
if a Hilbert space $\cl H$ is contractively contained in $H^2,$ then there exists a positive contraction 
$P\in B(H^2)$ such that $\cl H$ is the range space, $\cl R(P^{1/2}).$ 

On the other hand, given a positive contraction $P\in B(H^2)$ the range space $\cl R(P^{1/2}),$ as 
defined above, is always contractively contained in $H^2.$ 

Henceforth, given a positive contraction $P\in B(H^2),$ we shall denote the 
Hilbert space $\cl R(P^{1/2})$ by $\cl H(P)$ and the kernel function in it for a 
point $w\in \bb D$ by $k^P_w.$  
For the normalized kernel function $\frac{k_w^P}{\|k_w^P\|_P},$ 
we shall use $\widetilde{k}_w^P.$ 
Lastly, we note that 
$k^P_w=Pk_w$ for $w\in \bb D$ and $\|Px\|_P = \|P^{1/2}x\|$ for all
$x\in H^2.$ 

%For further details we refer to \cite{AFMP}. 

%First recall the definition of de-Branges space  $H(P)$ 
%given in Section \ref{intro}. 

%\begin{thm} The following are equivalent:
%\begin{itemize}
%\item[1.]every bounded Bessel sequence in $l^2$ can be partitioned into finitely many Riesz basic 
%sequences ({\bf{\em Feichtinger Conjecture}}),
%\item[2.] every Bessel sequence $\{\widetilde{k}_{z_i}^P\},$ of normalized kernel functions  in 
%each de-Branges space $H(P)$ 
%can partitioned into finitely many Riesz basic sequences.
%\item[3.]for each positive operator $P\in B(H^2),$ every bounded Bessel sequence 
%$\{P\widetilde{k}_{z_i}\}$ can be partitioned into finitely many Riesz basic sequences. 
%\end{itemize}
%Moreover, in the last two equivalences we can assume that $\{z_i\}$ satisfies the
%Blaschke condition, in fact we can assume a much more restrictive condition on $\{z_i\}$ 
%which is that $\{\widetilde{k}_{z_i}\}$ is a Riesz basic sequence in $H^2.$ 
%\end{thm}
\vspace{.2cm}

\noindent {\bf Proof of Theorem \ref{equi}.}
(i) implies (iii) is trivially true.
We now prove (iii) implies (ii). Let $P\in B(H^2)$ be a positive operator and  let $\{z_i\}_{i\in \bb N}$ be a sequence in $\bb D$ such that  
$\{P\widetilde{k}_{z_i}\}_{i\in \bb N}$ is a bounded Bessel sequence in $H^2$ with 
$\|P\widetilde{k}_{z_i}\|\ge \delta>0$ for all $i.$ Then
$T=P^2/\|P^2\|$ is a positive contraction in $B(H^2)$ and thus $\cl
H(T) = \cl R(T^{1/2})$ is a Hilbert space that is contractively contained in $H^2.$ Further, note that for fixed $i, \ j,$ 

\begin{eqnarray*}
\langle{\widetilde{k}_{z_j}^T,\widetilde{k}_{z_i}^T}\rangle_T 
&=&\left\langle{\frac{Tk_{z_j}}{\|Tk_{z_j}\|_T},\frac{Tk_{z_i}}{\|Tk_{z_i}\|_T}}\right\rangle_T\\
&=&\left\langle{\frac{T^{1/2}k_{z_j}}{\|T^{1/2}k_{z_j}\|},\frac{T^{1/2}k_{z_i}}{\|T^{1/2}k_{z_i}\|}}\right\rangle\\
&=&\left\langle{\frac{Pk_{z_j}}{\|Pk_{z_j}\|},\frac{Pk_{z_i}}{\|Pk_{z_i}\|}}\right\rangle\\
&=&\frac{\|k_{z_j}\|}{\|Pk_{z_j}\|}\langle{P\widetilde{k}_{z_j},P\widetilde{k}_{z_i}}\rangle
\frac{\|k_{z_i}\|}{\|Pk_{z_i}\|}.
\end{eqnarray*}

Hence,  
$$
\left(\langle{\widetilde{k}_{z_j}^T,\widetilde{k}_{z_i}^T}\rangle_T\right) 
=D\left(\langle{P\widetilde{k}_{z_j},P\widetilde{k}_{z_i}}\rangle\right)D^*,
$$
where $D\in B(\ell^2)$ is an invertible, diagonal operator with $i^{th}$ diagonal entry $\frac{\|k_{z_i}\|}{\|Pk_{z_i}\|},$ since  $P\in B(H^2)$ and $\|P k_{z_i}\|\ge \delta\|k_{z_i}\|$ for all $i.$ Then (i) of Proposition \ref{r_equi} implies that $\{\widetilde{k}_{z_i}^T\}_{i\in \bb N}$ is a Bessel sequence in $\cl H(T),$ since $\{P\widetilde{k}_{z_i}\}_{i\in \bb N}$ is a Bessel sequence in $H^2.$ Thus, by assuming (iii), we have that $\{\widetilde{k}_{z_i}^T\}_{i\in \bb N}$ satisfies the FC and hence we conclude that $\{P\widetilde{k}_{z_i}\}_{i\in \bb N}$ also satisfies the FC, by using (iii) of Proposition \ref{r_equi}. This completes the proof of (iii) implies (ii). 
  
Finally, we prove (ii) implies (i). Let $\{f_i\}_{i\in \bb N}$ be a bounded Bessel sequence in a Hilbert space $\cl H$ with $\|f_i\|\ge \delta >0$ for each $i.$ Then, $FF^*$ is a bounded, positive operator in $B(\ell^2),$ where $F:\cl H \to \ell^2$ is the analysis operator associated with $\{f_i\}_{i\in \bb N}.$ Also, $\langle{FF^*(e_i),e_i}\rangle= \|f_i\|^2 \ge \delta^2$ for all $i.$ Thus, by Theorem~\ref{st}, there 
exists a positive operator 
$P\in B(H^2)$ with $\|P\widetilde{k}_{z_i}\|\ge \delta$ for each $i,$ such that 
$$ FF^* = (\langle{f_j,f_i}\rangle)=\left(\langle P{\widetilde{k}_{z_j},P \widetilde{k}_{z_i}\rangle}\right),$$
where $\{\widetilde{k}_{z_i}\}_{i\in \bb N}$ is a frame sequence and hence is a Riesz basic sequence in $H^2.$
Since $FF^*\in B(\ell^2), \ \{P\widetilde{k}_{z_i}\}_{i\in \bb N}$ is a Bessel sequence. Also, $\{P\widetilde{k}_{z_i}\}_{i\in \bb N}$ is 
bounded. Thus by assuming (ii), we have that it satisfies the
FC. Hence, by (iii) of Proposition~\ref{r_equi}, $\{f_i\}_{i\in \bb N}$ also
satisfies the FC. This completes the proof of (ii) implies (i) and of
the theorem.  
 
Finally, note that the remark after the theorem follows immediately from the fact that in Theorem~\ref{st} we choose $\{z_i\}_{i\in \bb N}$ so that $\{\widetilde{k}_{z_i}\}_{i\in \bb N}$ is a Riesz basic sequence. 

%%%%%%%%%%%%%%%%%%%%%%%%%%%%%%%%%%%%%%%%%%%%%%%%%%%%%%%%%%%%%%%%%%%%%%%%%%%%%%%%%%%%%%%%%%%%%%%%%%%%%%%%%%%%%%%%%%%%%%%%%%%%%%%%%%%%%%%%%%%%

\section{Analysis Of New Equivalences}

We can easily verify that statement (ii) of Theorem~\ref{equi} can be reduced to the case of 
positive operators which are contractions. Thus, Theorem~\ref{equi} motivates the study of sequences 
$\{P\widetilde{k}_{z_i}\}_{i\in \bb N}, \ \{\widetilde{k}^P_{z_i}\}_{i\in \bb N},$ where $P\in B(H^2)$ is a positive contraction and 
$\{\widetilde{k}_{z_i}\}_{i\in \bb N}, \ \{\widetilde{k}^P_{z_i}\}_{i\in \bb N}$ are sequences of 
normalized kernel functions in $H^2$ and $\cl H(P),$ respectively. 
By considering positive operators and kernel functions we have much more 
structure to exploit and thereby we can expect some interesting and fruitful research in this 
direction. The theorem suggests that it might not be easy to make any general 
statement about the whole family of these sequences. Instead we focus on some particular families of positive operators and investigate the FC for the corresponding sequences. In this direction we have the following results.

It is elementary to see that the FC holds for $\cl H(P)$ when $P$ is a positive, invertible operator in $B(H^2).$  
Note in this case $\cl H(P)=H^2$ and the two norms are equivalent. Thus, $\{P\widetilde{k}_{z_i}\}_{i\in \bb N}$ is 
a Bessel (frame or Riesz basic) sequence iff $\{\widetilde{k}^P_{z_i}\}_{i\in \bb N}$ is Bessel 
(frame or Riesz basic) sequence iff $\{ \widetilde{k}_{z_i}\}_{i\in \bb N}$ is a Bessel(frame or Riesz basic) sequence. Thus,  every Bessel sequence of 
normalized kernel functions in $\cl H(P)$ satisfies the FC, since the FCKF holds for the Hardy space and hence the space $\cl H(P)$ satisfies the FCKF.

We now focus on some orthogonal projections. If $P$ is an orthogonal projection, then $\cl H(P)$ is just the range of $P$ and the norm is just the usual Hardy space norm.
 In 
\cite{BD}, Baranov and Dyakonov have considered the FC for what are often called de Branges spaces, that is spaces of the form $H^2\ominus \phi H^2$ with some conditions 
on $\phi$ and proved the following two theorems. 

\begin{thm}[Branov-Dyakonov]\label{BD1}Let $\phi$ be an inner function. If $\{z_i\}_{i\in \bb N}$ is a sequence in $\bb D$ such that 
${\rm sup}_i|\phi(z_i)|<1$ for all  $i,$ then the corresponding sequence of normalized kernel functions in $H^2\ominus \phi H^2$ satisfies the FC.  
\end{thm}

%Every Bessel sequence of normalized kernel functions $\{\widetilde{k}^P_{z_i}\}$ 
%in $H^2\ominus \phi H^2$ satisfies FC, where  $P$ is the projection onto $H^2\ominus \phi H^2,$ and 
%${\rm sup}_i|\phi(z_i)|<1$ for all  $i.$

The second theorem of Baranov and Dyakonov uses one-component inner functions. An inner function $\phi$ is said to be an {\bf one-component} inner function if the set $\{z:|\phi(z)|<\epsilon\}$ is connected for some $\epsilon\in (0,1).$
  
\begin{thm}[Baranov-Dyakonov]\label{BD2}Assume that $\phi$ is a one-component inner function. Then every Bessel 
sequence of normalized 
kernel functions in $H^2\ominus \phi H^2$ satisfies the FC.
\end{thm}

Note that given an inner function $\phi,$ the model space 
$H^2\ominus \phi H^2$ is the de~Branges space $\cl H(P),$ where 
$P\in B(H^2)$ is the orthogonal projection onto $H^2\ominus \phi H^2.$ 
Hence, the above theorems of 
Baranov and Dyakonov analyzes the class of de~Branges spaces $\cl H(P)$ for FCKF, where $P$ belongs to the 
family of projections onto $H^2\ominus \phi H^2$ and $\phi$ is an inner function with 
properties, as stated 
in Theorem \ref{BD1} and \ref{BD2}. In particular, their second theorem proves that when $\phi$ is an one-component inner function and $P$ is the orthogonal projection onto 
$H^2\ominus \phi H^2,$ then the de Branges space $\cl H(P)$ satisfies FCKF.  
 
The result for the complementary projection is elementary.

\begin{thm}\label{inner3} Let $\phi$ be an inner function
and let $P_\phi$ be the orthogonal projection onto $\phi H^2.$ Then the space $\cl H(P_\phi)$ satisfies the FCKF. 
\end{thm}

\begin{proof}Let $\{\widetilde{k}_{z_i}^{P_{\phi}}\}_{i\in \bb N}, \ \{z_i\}_{i\in \bb N}\subseteq \bb D$ be a Bessel sequence in $\cl H(P_\phi).$ To prove that this sequence satisfies the FC, we first observe that $P_\phi=T_\phi{T_\phi}^*,$ where $T_\phi$ is the Toeplitz operator with 
symbol $\phi$ and ${T_\phi}^*k_{z_i}=\overline{\phi(z_i)}k_{z_i}.$ 
Also, $\cl H(P_\phi)$ coincides with the range of $P_{\phi}$ and the two norms are equal, since $P_{\phi}$ is an orthogonal projection. To simplify notation, we set $P=P_{\phi}.$ Then, 
$$
{\langle{\widetilde{k}_{z_j}^P,\widetilde{k}_{z_i}^P}\rangle}_P
= \left\langle{\frac{Pk_{z_j}}{\|Pk_{z_j}\|},\frac{Pk_{z_i}}{\|Pk_{z_i}\|}}\right\rangle
=\frac{\overline{\phi(z_j)}}{|\phi(z_j)|}\langle{\widetilde{k}_{z_j},\widetilde{k}_{z_i}}\rangle\frac{\phi(z_i)}{|\phi(z_i)|}
$$
Hence,
$$
\left({\langle{\widetilde{k}_{z_j}^P,\widetilde{k}_{z_i}^P}\rangle}_P\right) = D \left(\langle{\widetilde{k}_{z_j},\widetilde{k}_{z_i}}\rangle \right)D^*,
$$  
where $D\in B(\ell^2)$ is an invertible, diagonal operator with $i^{th}$ diagonal entry $\frac{\phi(z_i)}{|\phi(z_i)|}.$ 
Finally, using Proposition \ref{FC-equi} and Theorem \ref{NM}, we conclude that $\{\widetilde{k}_{z_i}^{P_\phi}\}_{i\in \bb N}$ satisfies the FC and hence $\cl H(P_\phi)$ satisfies the FCKF. 
\end{proof}

By taking a closer look at the proof of (iii) implies (ii) in Theorem \ref{equi}, we notice that in order to prove that for an orthogonal projection $P$ a bounded Bessel sequence $\{P\widetilde{k}_{z_i}\}_{i\in \bb N}$ satisfies the FC, all we need is that the corresponding sequence $\{\widetilde{k}_{z_i}^P\}_{i\in \bb N}$ in $\cl H(P),$ for the same $P,$ satisfies the FC. As an immediate consequence we get the following result.  

\begin{thm}\label{inner1}Let $\phi$ be an inner function
and let $P_\phi$ be the orthogonal projection onto $\phi H^2.$ If  
$\{z_i\}_{i\in \bb N}$ is a sequence in $\bb D$ such that  $\{P_\phi\widetilde{k}_{z_i}\}_{i\in \bb N}$ is a bounded Bessel sequence in $H^2$, then $\{P_\phi\widetilde{k}_{z_i}\}_{i\in \bb N}$ satisfies the FC. 
\end{thm}

%\begin{thm}\label{inner1}Let $\phi$ be an inner function
%and let $P_\phi$ be the orthogonal projection onto $\phi H^2.$ If  
%$\{P_\phi\widetilde{k}_{z_i}\}$ is a bounded Bessel sequence for some sequence $\{z_i\}$ in $\bb D,$ then it satisfies the FC. 
%\end{thm}

%Next, we present a different proof of Theorem \ref{inner1}. This detour is worth looking at, as it not only motivates a generalization of Theorem \ref{inner1}, but also reveals some interesting properties of the sequences $\{P_\phi\widetilde{k}_{z_i}\}.$ We start with the following proposition.
 
These last two results allow us to draw some conclusions about spaces that contain the image of an inner function, that look analogous to the theorems of Baranov and Dyakanov and require bounds away from 0 instead of away from 1. First we need a preliminary result.

\begin{prop}\label{inner}Let $\phi$ be an inner function
and let $P_\phi$ be the orthogonal projection onto $\phi H^2.$
Then for every sequence $\{z_i\}_{i\in \bb N}\subseteq\bb D$ such that there exists a $\delta>0$ with $|\phi(z_i)|\ge \delta$ for all $i,$ the following hold true:
%Let $\{z_i\}$ be a sequence in $\bb D$ and let $\phi$ be an inner 
%function such that there exists a constant $\delta>0$ with 
%$|\phi(z_i)|\ge \delta$ for all $i.$ 
%If $P_\phi$ 
%is the orthogonal projection onto $\phi H^2,$ then:
\begin{enumerate}
 \item[(i)] for each $i, \ {\delta}\|k_{z_i}\|\le \|P_\phi k_{z_i}\|\le \|k_{z_i}\|,$
\item[(ii)] $\{P_\phi\widetilde{k}_{z_i}\}_{i\in \bb N}$ is a Bessel sequence iff  $\{\widetilde{k}_{z_i}\}_{i\in \bb N}$ 
is a Bessel sequence,
\item[(iii)]$\{P_\phi \widetilde{k}_{z_i}\}_{i\in \bb N}$ is a frame sequence iff it is a Riesz basic sequence,
\item[(iv)] $\{P_\phi\widetilde{k}_{z_i}\}_{i\in \bb N}$ is a Riesz basic sequence iff $\{\widetilde{k}_{z_i}\}_{i\in \bb N}$ 
is a Riesz basic sequence.
\end{enumerate}
\end{prop}

\begin{proof}Let $\{z_i\}_{i\in \bb N}$ be a sequence in $\bb D$ and let $\delta>0$ be a constant such that $|\phi(z_i)|\ge \delta$ for all $i.$  
   
As noted earlier, $P_\phi=T_\phi{T_\phi}^*,$ where $T_\phi$ is the Toeplitz operator with 
symbol $\phi,$ and ${T_\phi}^*k_{z_i}=\overline{\phi(z_i)}k_{z_i}.$ 
Thus,
\begin{equation}
\langle{P_\phi\widetilde{k}_{z_j}, P_\phi\widetilde{k}_{z_i}}\rangle
= \langle{{T_\phi}^*\widetilde{k}_{z_j}, {T_\phi}^*\widetilde{k}_{z_i}}\rangle
= \overline{\phi(z_j)}\langle{\widetilde{k}_{z_j}, \widetilde{k}_{z_i}}\rangle\phi(z_i)
\end{equation}
Hence, 
\begin{equation}\label{D}
\left(\langle{P_\phi\widetilde{k}_{z_j}, P_\phi\widetilde{k}_{z_i}}\rangle\right) = D\left(\langle{\widetilde{k}_{z_j}, \widetilde{k}_{z_i}}\rangle\right)D^*,
\end{equation}
where $D\in B(\ell^2)$ is an invertible, diagonal operator with $\phi(z_i)$ as the $i^{th}$ diagonal entry, since $\delta \le |\phi(z_i)| \le 1$ for all $i.$ 
  
Clearly, (i) follows from Equation (2) and (ii), (iii) and (iv) follows from Equation (3), using Proposition \ref{r_equi} and Theorem \ref{fr}. 
\end{proof}

%\begin{thm}Let $\phi$ be an inner function
%and let $P_\phi$ be the orthogonal projection onto $\phi H^2.$ If  
%$\{P_\phi\widetilde{k}_{z_i}\}$ is a bounded Bessel sequence for some sequence $\{z_i\}$ in $\bb D,$ then it satisfies the FC. 
%Let $\{z_i\}\subseteq \bb D$ and let $\phi$be an inner function. 
%Then every bounded Bessel sequence $\{P_\phi\widetilde{k}_{z_i}\}$ satisfy the FCC, where 
%$P_{\phi}$ is the projection onto $\phi H^2.$
%\end{thm}

%{\bf Alternate proof of Theorem \ref{inner1}.} Let $\{z_i\}$ be a sequence in $\bb D$ such that $\{P_\phi\widetilde{k}_{z_i}\}$ is a bounded Bessel sequence. Then, there 
%exists a constant $\delta >0$ such that for each $i,
%\ \delta \|k_{z_i}\|\le \|P_\phi k_{z_i}\|\le \|k_{z_i}\|.$ Now as obtained in propos%ition \ref{inner}, we get 
%$$
%\|P_\phi k_{z_i}\| = |\phi(z_i)| \|k_{z_i}\|
%$$
%Thus, $|\phi(z_i)| \ge \delta$ for all $i.$ Then by (iii) of Proposition \ref{inner},  
%the sequence $\{\widetilde{k}_{z_i}\}$ is a Bessel sequence and thus satisfies the 
%FC, using Theorem \ref{NM}. Hence by part (iv) of Proposition \ref{inner}, $\{P_{\phi%}\widetilde{k}_{z_i}\}$ satisfies the FC.\\  
%%\begin{remark}Note that the only facts used to prove Proposition \ref{inner} 
%were, $P$ is a projection, $\phi H^2\subseteq \cl R(P),$ and $|\phi(z_i)|\ge \delta, \ \forall \ i.$ 
%Hence, we can make the following general statement. 
%\end{remark}
 
We can generalize Proposition~\ref{inner} and Theorem~\ref{inner1} as follows. 
Given an operator $T\in B(H^2)$ we shall denote the range of $T$ by $Ran(T).$

\begin{prop}\label{inner4}Let $P\in B(H^2)$ 
be an orthogonal projection. Given a sequence $\{z_i\}_{i\in \bb N}$ in $\bb D,$ if there exists an inner function $\phi$ such that $|\phi(z_i)|\ge \delta$ for all $i$ and $\phi H^2\subseteq Ran(P),$ then:  
\begin{enumerate}
 \item[(i)] for each $i, \ {\delta}\|k_{z_i}\|\le \|Pk_{z_i}\|\le \|k_{z_i}\|,$
\item[(ii)]$\{P\widetilde{k}_{z_i}\}_{i\in \bb N}$ is a Bessel sequence iff $\{\widetilde{k}_{z_i}\}_{i\in \bb N}$ 
is a Bessel sequence,
\item[(iii)]$\{P\widetilde{k}_{z_i}\}_{i\in \bb N}$ is a frame sequence iff it is a Riesz basic sequence,
\item[(iv)]$\{P\widetilde{k}_{z_i}\}_{i\in \bb N}$ is a Riesz basic sequence iff $\{\widetilde{k}_{z_i}\}_{i\in \bb N}$ 
is a Riesz basic sequence,
\end{enumerate} 
\end{prop}

%\begin{prop}Let $\{z_i\}$ be a sequence in $\bb D$ and let $\phi$ be 
%an inner function such that there exists a constant $\delta>0$ with 
%$|\phi(z_i)|\ge \delta.$ If $P\in B(H^2)$ 
%is a projection such that $\phi H^2\subseteq Ran(P).$ Then: 
%\begin{enumerate}
% \item[(i)] for each $i, \ {\delta}\|k_{z_i}\|\le \|Pk_{z_i}\|\le \|k_{z_i}\|,$
%\item[(ii)]$\{P\widetilde{k}_{z_i}\}$ is a Bessel sequence iff $\{\widetilde{k}_{z_i}\}$ 
%is a Bessel sequence,
%\item[(iii)]$\{P\widetilde{k}_{z_i}\}$ is a frame sequence iff it is a Riesz basic sequence,
%\item[(iv)]$\{P\widetilde{k}_{z_i}\}$ is a Riesz basic sequence iff $\{\widetilde{k}_{z_i}\}$ 
%is a Riesz basic sequence,
%\end{enumerate} 
%\end{prop}

\begin{proof}Let $P_{\phi}$ denote the orthogonal projection onto 
$\phi H^2.$ Then $P_{\phi}\le P$ and thus, $\|P_{\phi}k_{z_i}\| \le \|Pk_{z_i}\| \le \|k_{z_i}\|.$ This proves (i), since $\delta \|k_{z_i}\| \le \|P_{\phi}k_{z_i}\|.$
 
To prove (ii), we first note that for any $x\in H^2,$ 
$$ \langle{x, P_{\phi}\widetilde{k}_{z_i}}\rangle = \langle{P_{\phi}x, P\widetilde{k}_{z_i}}\rangle  
$$
Thus, if $\{P\widetilde{k}_{z_i}\}_{i\in \bb N}$ is a Bessel sequence, then $\{P_{\phi}\widetilde{k}_{z_i}\}_{i\in \bb N}$ is a Bessel sequence and hence, $\{\widetilde{k}_{z_i}\}$ is a Bessel sequence, using Proposition \ref{inner}. The other implication follows trivially from the fact that $P\in B(H^2).$  

We shall now prove (iii) and (iv). Note that if $\{\widetilde{k}_{z_i}\}_{i\in \bb N}$ is a Bessel sequence, then 
\begin{equation}\label{P}
\left(\langle{P_\phi\widetilde{k}_{z_j}, P_\phi\widetilde{k}_{z_i}}\rangle\right) \le  \left(\langle{P\widetilde{k}_{z_j}, P\widetilde{k}_{z_i}}\rangle\right) \le \left(\langle{\widetilde{k}_{z_j},\widetilde{k}_{z_i}}\rangle\right) 
\end{equation}
To prove (iii), we first assume that $\{P\widetilde{k}_{z_i}\}_{i\in \bb N}$ is a frame sequence. Then it is a Bessel sequence and thus $\{\widetilde{k}_{z_i}\}_{i\in \bb N}$ is also a Bessel sequence. So, Equation 
(\ref{P}) holds and we get $F_{\phi}F_{\phi}^*\le F_PF_P^* \le FF^*,$ where $F_{\phi}, \ F_P$ and $F$ are the analysis operators corresponding to the sequences $\{P_\phi\widetilde{k}_{z_i}\}_{i\in \bb N}, \ \{P\widetilde{k}_{z_i}\}_{i\in \bb N}$ and $\{\widetilde{k}_{z_i}\}_{i\in \bb N},$ respectively. We claim that $Ker(F_P^*) = \{0\}.$ To accomplish the claim, we first note that  $Ker(F_P^*) \subseteq Ker(F_{\phi}^*),$ 
using Equation (\ref{P}). Further, using the same ideas as used in the proof of Theorem \ref{fr}, we conclude that $Ker(F^*) = \{0\}.$ This implies that $Ker(F_{\phi}^*) = \{0\},$ using Equation (\ref{D}). Thus it follows that $Ker(F_P^*) = \{0\}.$ 

Lastly, $F_P^*$ is onto as well, since 
$\{P\widetilde{k}_{z_i}\}_{i\in \bb N}$ is a frame sequence. Therefore $F_P^*$ and hence $F_P$ is an invertible operator, which implies that the sequence $\{P\widetilde{k}_{z_i}\}_{i\in \bb N}$ is a Riesz basic sequence. The other  implication in (iii) is trivially true.        
%Now, we invoke Douglas' factorization theorem to get a surjective, bounded operator 
%$T:\overline{\span\{P\widetilde{k}_{z_i}:i\in \bb N\}} \to \overline{\span\{P_\phi\widetilde{k}_{z_i}:i\in \bb N\}}$ such that $TF_P^*=F_{\phi}^*.$ 
%Thus $F_PT^*=F_{\phi}$ and hence $T^*$ is bounded below on $\overline{\span\{P_\phi\widetilde{k}_{z_i}:i\in \bb N\}},$ since $F_{\phi}$ 
%is an isometry. Also, $F_P$ is bounded below on $\overline{\span\{P\widetilde{k}_{z_i}:i\in \bb N\}},$ since $\{P\widetilde{k}_{z_i}\}_{i\in \bb N}$ is a frame sequence. Thus, $F_{\phi}^*$ is bounded below on $\overline{\span\{P_\phi\widetilde{k}_{z_i}:i\in \bb N\}}.$ Therefore, $\{P_{\phi}\widetilde{k}_{z_i}\}_{i\in \bb N}$ is a frame sequence and hence is a Riesz basic sequence, by (iii) of Proposition \ref{inner}. 
%Thus, by Proposition \ref{riesz2} there exists a constant $c>0$ such that 
%$$
%c I_{l^2} \le F_{\phi}F_{\phi}^* = \left(\langle{P_\phi\widetilde{k}_{z_j}, P_\phi\widetilde{k}_{z_i}}\rangle\right) \le  \left(\langle{P\widetilde{k}_{z_j}, P\widetilde{k}_{z_i}}\rangle\right)  
%= F_PF_P^*
%$$
%This implies that $\{P\widetilde{k}_{z_i}\}$ is a Riesz basic sequence, using equation (\ref{P}) and Proposition \ref{riesz2}.  
  
Finally, (iv) follows from Equation (\ref{P}), using (iv) of Proposition \ref{inner} together with Proposition \ref{riesz2}.   
\end{proof}

\begin{thm}\label{inner5}Let $P\in B(H^2)$ 
be an orthogonal projection. Let $\{z_i\}_{i\in \bb N}$ be a sequence in $\bb D$ such that there exists an inner function $\phi$ with $|\phi(z_i)|\ge \delta$ for all $i$ and $\phi H^2\subseteq Ran(P).$ If $\{P\widetilde{k}_{z_i}\}_{i\in \bb N}$ is a Bessel sequence, then it satisfies the FC.
\end{thm}

\begin{proof}The proof follows immediately from Proposition~\ref{inner4}.
\end{proof}

To illustrate the above theorem we would like to mention a few examples where the hypotheses of Proposition~\ref{inner4} are satisfied, and so the conclusions of Proposition~\ref{inner4} and Theorem~\ref{inner5} 
apply. 

\begin{ex}\label{ex1}Let $P\in B(H^2)$ be the orthogonal projection onto the closed linear span of 
$\{z^j:j\ne j_1,\dots, j_n\}, \ j_1<\cdots<j_n,$ and let  $\{z_i\}_{i\in \bb N}$ be a sequence in $\bb D$ such that there exists a constant
$\delta>0$ with $|z_i|\ge \delta$ for all $i.$ Then, $\phi(z)=z^{j_n+1}$ is an inner function, $\phi H^2\subseteq Ran(P)$ and $|\phi(z_i)|\ge {\delta}^{j_n+1}$ for all  $i.$ 
Hence $P$ and $\{z_i\}_{i\in \bb N}$ satisfy the conditions of Proposition \ref{inner4}.
\end{ex}

\begin{ex}\label{ex2}Given an inner function $\phi, \ [\bb C+ \phi H^2]$ denotes the span of 
$\bb C$ and $\phi H^2$ in $H^2.$ These spaces were first introduced in \cite{Rag}. Let $\{z_i\}_{i\in \bb N}$ be a sequence in $\bb D$ 
and let $\phi$ be an inner function such that $|\phi(z_i)|\ge \delta>0$ for all  $i.$ Then the orthogonal 
projection $P$ onto $[\bb C+H^2_\phi]$ and $\{z_i\}_{i\in \bb N}$ satisfy the conditions of Proposition \ref{inner4}. 
\end{ex}

\begin{ex}\label{ex3}Let $P\in B(H^2)$ be an orthogonal projection such that the kernel of $P$ is spanned by $n$ inner functions $\phi_1, \dots, \phi_n$ . Then 
$\phi = z\phi_1\cdots\phi_n$ is an inner function and $\phi H^2\subseteq Ran(P).$
Now, if $\{z_i\}_{i\in \bb N}$ is a sequence in $\bb D$ such that there exists a constant $\delta>0$ with 
$|z_i|\ge \delta, \ |\phi_{k}(z_i)|\ge \delta$ for all $i$ and for all $k,$ then $|\phi(z_i)|\ge {\delta}^{n+1}$ for all $i.$ Thus $P$ and $\{z_i\}_{i\in \bb N}$ satisfy the conditions of Proposition 
\ref{inner4}.
\end{ex}

\begin{remark}If $\phi$ is a finite Blaschke product and $\{z_i\}_{i\in \bb N}\subseteq \bb D$ converges to 1, 
then the condition $|\phi(z_i)|\ge \delta,$ follows automatically for all, but finitely 
many $z_i's$. Because, the zeroes of $\phi$ 
lies in the set $\{z:|z|<r\}$ for some $r>0,$ and $|z_i|$ converges to 1. Hence, when $\phi$ is 
a finite Blaschke and $\{|z_i|\}$ converges to 1, then the bounded below assumption on $\phi$ 
in Proposition \ref{inner4} and Theorem \ref{inner5}  is redundant.
\end{remark}
 
\begin{remark}\label{inner7} For the case, when $P$ is an orthogonal projection, the Hilbert space $\cl H(P)$ coincides with $Ran(P).$
Further, in this case, if there exists a constant 
$\delta>0$ such that $\ {\delta}\|k_{z_i}\|\le \|Pk_{z_i}\|\le \|k_{z_i}\|$ for all  $i,$ then 
$\{P\widetilde{k}_{z_i}\}_{i\in \bb N}$ is a Bessel (frame or Riesz basic) sequence iff $\{\widetilde{k}^P_{z_i}\}_{i\in \bb N}$ 
is Bessel (frame or Riesz basic) sequence. Hence, if $P$ is an orthogonal projection and $\{z_i\}_{i\in \bb N}$ is a sequence in $\bb D$ such that there exists an inner function $\phi$ satisfying the condition of Theorem \ref{inner5}, then whenever $\{\widetilde{k}_{z_i}^P\}_{i\in \bb N}$ is a  
Bessel sequence in $\cl H(P)$, it  satisfies the FC.
\end{remark}
%\begin{remark} For the case, when $P$ is an orhotgonal projection, the Hilbert space $\cl H(P)$ coincides with $Ran(P).$
%Also, if $\{P\widetilde{k}_{z_i}\}$ is a bounded sequence, then there exists a constant 
%$\delta>0$ such that $\ {\delta}\|k_{z_i}\|\le \|Pk_{z_i}\|\le \|k_{z_i}\|$ for all  $i.$ Thus, 
%$\{P\widetilde{k}_{z_i}\}$ is a Bessel (frame or Riesz basic) sequence iff $\{\widetilde{k}^P_{z_i}\}$ 
%is Bessel (frame or Riesz basic) sequence. Hence if $P$ is as in Theorem \ref{inner5}, then every 
%Bessel sequence of normalized kernel functions in $\cl H(P)$ satisfies FCC.
%\end{remark}
%We saw that the category, where $P$ is invertible satisfies the FC. Further, the case when $P$ 
%is a projection with $P\ge P_{\phi}$ for some inner function $\phi,$ also satisfy the  FC, 
%but under a mild condition on $\phi.$ We still have lot more to explore as we don't know what 
%happens to the case of a general projection in $B(H^2).$ It seems that the case corresponding 
%to projections with finite co-dimension should not be too 
%hard but, we have not reached there yet. 

%Every Bessel sequence of normalized kernel functions $\{\widetilde{k}^P_{z_i}\}$ 
%in $H^2\ominus \phi H^2$ satisfies FC, where  $P$ is the projection onto $H^2\ominus \phi H^2,$ and 
%${\rm sup}_i|\phi(z_i)|<1$ for all  $i.$

Our last results focus on the weighted Hardy spaces on the unit disk. We shall briefly define these spaces here, for more details we refer to \cite{Shi}.

Let $\{\beta_n\}$ be a sequence of positive numbers with $R={\rm \liminf} (\beta_n)^{1/n} >0.$ Then the set $\{\sum_na_nz^n: \sum_n\beta_n^2|a_n|^2<\infty\}$ is a reproducing kernel Hilbert space on the disk of radius $R$ with norm $\|\sum_na_nz^n\|_{\beta}^2=\sum_n\beta_n^2|a_n|^2$ and reproducing kernel $K_{\beta}(z,w)=\sum_n\frac{\bar{w}^nz^n}{{\beta_n}^2}.$ This Hilbert space is called a weighted Hardy space and is denoted by $H^2(\beta).$ If we let $P\in B(H^2)$ be a positive, diagonal contraction with $n^{th}$ diagonal entry $p_n.$ Then the space $\cl H(P)$ coincides with the weighted Hardy space $H^2(\beta),$ where $\beta_n=\frac{1}{\sqrt{p_n}}$ for every $n$ and the functions in $H^2(\beta)$ are restricted to the unit disk $\bb D.$

%The diagonal operators are better understood and easy to work with, as compared to positive operators. 
%So, the natural question which arises at this point is, is it enough to consider positive, diagonal operators in statements (ii) and (iii) of our main theorem? This motivates the following results. 

\begin{prop}\label{diagonal}Let $P\in B(H^2)$ be a positive operator and $D\in B(H^2)$ be a 
positive, diagonal operator such that $\alpha D\le P \le \beta D$ for some $\alpha,\ \beta>0.$ Then:
\begin{enumerate}
\item [(i)]$\{P^{1/2}\widetilde{k}_{z_i}\}_{i\in \bb N}$ is a Bessel sequence iff $\{D^{1/2}\widetilde{k}_{z_i}\}_{i\in \bb N}$ is a Bessel sequence,
\item[(ii)] $\{P^{1/2}\widetilde{k}_{z_i}\}_{i\in \bb N}$ is a frame sequence iff $\{D^{1/2}\widetilde{k}_{z_i}\}_{i\in \bb N}$ is a frame sequence,
\item[(iii)] $\{P^{1/2}\widetilde{k}_{z_i}\}_{i\in \bb N}$ is a Riesz basic sequence iff $\{D^{1/2}\widetilde{k}_{z_i}\}_{i\in \bb N}$ is a Riesz basic sequence.
\end{enumerate}
\end{prop}
\begin{proof}As a direct consequence of the given inequalities we get
$$
\alpha\|D^{1/2}x\|^2\le \|P^{1/2}x\|^2\le\beta \|D^{1/2}x\|^2
$$
for every $x\in H^2.$ Using this we can easily verify that (i), (ii) and (iii) holds true. 
\end{proof}

\begin{thm}Let $P$ be a positive operator and $D\in B(H^2)$ be a positive, diagonal operator
such that $\alpha D\le P \le \beta D$ for some $\alpha,\ \beta>0.$ Then  $\{P^{1/2}\widetilde{k}_{z_i}\}_{i\in \bb N}$ satisfies the FC iff  $\{D^{1/2}\widetilde{k}_{z_i}\}_{i\in \bb N}$ satisfies the FC.
\end{thm}
\begin{proof}This follows immediately from Proposition \ref{diagonal}.
\end{proof}

We conclude with a brief summary. As mentioned earlier, in the case of an orthogonal projection $P,$ in order to prove that a bounded Bessel sequence $\{P\widetilde{k}_{z_i}\}_{i\in \bb N}$ satisfies the FC, it is enough to have that the corresponding sequence $\{\widetilde{k}_{z_i}^P\}_{i\in \bb N}$ in $\cl H(P)$ satisfies the FC. Hence, when $P$ is an orthogonal projection we shall only mention the results about sequences of normalized kernel functions in $\cl H(P).$ 
 
\begin{itemize}
\item $H^2$ satisfies the FCKF (Nikolski, 2006). Note that $H^2=\cl H(P)$ with $P=I,$ the identity operator.  

\item Given an inner function $\phi,$ every Bessel sequence $\{\widetilde{k}_{z_i}^P\}$ in $\cl H(P)$ such that ${\sup_i}|\phi(z_i)|<1$ satisfies the FCKF, where $P$ is the orthogonal projection onto $H^2\ominus \phi H^2$(Baranov and Dyakonov, 2009).

\item Given a one-component inner function $\phi,$ the de Branges $\cl H(P)$ satisfies the FCKF, where $P$ is the orthogonal projection onto $H^2\ominus \phi H^2$(Baranov and Dyakonov, 2009).

\item Given a positive, invertible operator $P\in B(H^2),$ every Bessel sequence $\{P\widetilde{k}_{z_i}\}$ satisfies the FC.

\item Given a positive, invertible operator $P\in B(H^2),$ the space 
$\cl H(P)$ satisfies the FCKF.

\item Given an inner function $\phi,$ the space $\cl H(P_\phi)$ satisfies the FCKF, where $P_{\phi}$ is the orthogonal projection onto $\phi H^2$ (Theorem \ref{inner3}).

\item Given an orthogonal projection $P,$ if $\{\widetilde{k}_{z_i}^P\}$ is a Bessel sequence in $\cl H(P)$ such that there exists an inner function $\phi$ with ${\rm inf_i}|\phi(z_i)|>0$ and $\phi H^2$ contained in the range of $P,$ then $\{\widetilde{k}_{z_i}^P\}_{i\in \bb N}$ satisfies the FC (Remark \ref{inner7}).
\end{itemize}

\end{document}